\documentclass[11pt,leqno]{amsart}
\usepackage{amsmath,amssymb,amsfonts,mathrsfs, centernot}
\usepackage{amsthm}
\usepackage{enumerate}

\newtheorem{introthm}{Theorem}

\newtheorem{introlem}[introthm]{Lemma}

\newtheorem{theorem}{Theorem}[section]
\newtheorem{corollary}[theorem]{Corollary}
\newtheorem{definition}[theorem]{Definition}
\newtheorem{lemma}[theorem]{Lemma}

\theoremstyle{definition}

\newtheorem{example}[theorem]{Example}
\newtheorem{remark}[theorem]{Remark}

\newcommand{\N}{\mathbb{N}}

\newcommand{\R}{\mathbb{R}}

\newcommand{\wass}{\mathop{\mathscr{W}_2}\nolimits}

\begin{document}
\pagebreak


\title{Isometric rigidity of $L^2$-spaces with manifold targets}
 
\author{David Lenze}

\address
  {Karlsruher Institut f\"ur Technologie\\ Fakult\"at f\"ur Mathematik \\
Englerstr. 2 \\
76131 Karlsruhe,
Germany}
\email{david.lenze@kit.edu}

\begin{abstract}
We describe the isometry group of $L^2(\Omega, M)$ for Riemannian manifolds $M$ of dimension at least two with irreducible universal cover. We establish a rigidity result for the isometries of these spaces: any isometry arises from an automorphism of $\Omega$ and a family of isometries of $M$, distinguishing these spaces from the classical $L^2(\Omega)=L^2(\Omega,\R)$. Additionally, we prove that these spaces lack irreducible factors and that two such spaces are isometric if and only if the underlying manifolds are.
\end{abstract}
\maketitle

\renewcommand{\theequation}{\arabic{section}.\arabic{equation}}
\pagenumbering{arabic}

\section{Introduction}
\subsection{Main results}
	
The space $L^2(\Omega,X)$ consists of measurable, essentially separably valued functions $f:\Omega \to X$ from a finite measure space $(\Omega,\mu)$ to a metric space $(X,d)$ such that 
\[
\int_\Omega d^2(f(\omega),x)\,d\mu(\omega)<\infty,
\]

for some (and hence any) $x\in X$. We naturally equip $L^2(\Omega,X)$ with the metric 
$
d_{L^2}(f,g) = \left(\int_\Omega d^2(f(\omega),g(\omega))\, d\mu(\omega)\right)^{1/2}.
$

These spaces have found applications in various directions, including recently in the study of $L^2$-geometries on spaces of Riemannian metrics and related problems, as seen in Cavallucci \cite{MR4579388}, Cavallucci and Su \cite{MR4624071}, and B\"ohm, Buttsworth, and Clarke \cite{MR4780984}.

Let $(\Omega_1, \mu_1)$ and $(\Omega_2, \mu_2)$ be measure spaces. A bijection $\varphi: \Omega_1 \to \Omega_2$ is a \textit{strict isomorphism} if both $\varphi$ and $\varphi^{-1}$ are measure-preserving. An \textit{isomorphism} is a map that restricts to a strict isomorphism on full-measure subspaces. An isomorphism from a measure space to itself is an \textit{automorphism}, and the group of automorphisms of $(\Omega, \mu)$ is denoted $\operatorname{Aut}(\Omega)$. Furthermore $L^2(\Omega, \operatorname{Isom}(X))$ is the group, under pointwise composition, of functions $\rho: \Omega \to \operatorname{Isom}(X)$ such that for all $x \in X$, $\omega \mapsto \rho(\omega)(x)$ lies in $L^2(\Omega, X)$.
Both $\operatorname{Aut}(\Omega)$ and $L^2(\Omega, \operatorname{Isom}(X))$ are subgroups of the isometry group of $L^2(\Omega, X)$:
\begin{enumerate}
    \item For $\varphi \in \text{Aut}(\Omega)$, the map $f \mapsto f \circ \varphi$ is an isometry of $L^2(\Omega,X)$.
    \item For $\rho \in L^2(\Omega, \operatorname{Isom}(X))$, the map $f \mapsto \gamma_\rho(f)$, where $\gamma_\rho(f)(\omega) := \rho(\omega)(f(\omega))$ for all $\omega \in \Omega$, is an isometry of $L^2(\Omega,X)$.
\end{enumerate}

In this paper, we show that for a complete Riemannian manifold $M$ of dimension at least two with irreducible universal cover, the isometries of $L^2(\Omega,M)$ are rigid in the sense that every isometry is a composition of an isometry of type (1) with one of type (2). Indeed we establish:
\begin{introthm} \label{semi-main} Let $(\Omega,\mu)$ be a standard probability space and $M$ be a complete Riemannian manifold with irreducible universal covering and dim$(M)\geq 2$. Then $ \text{Isom}(L^2(\Omega,M))= L^2(\Omega,\text{Isom}(M)) \rtimes \text{Aut}(\Omega).$
\end{introthm}

Assuming $\Omega$ to be a standard probability space is a natural choice, encompassing a wide range of interesting cases, such as the unit interval equipped with the Lebesgue measure or the discrete probability space $\Omega=\{1,...,n\}$ with equal weights. In the latter case, $(L^2(\Omega,M),d_{L^2})\cong (M^n,\tfrac{1}{\sqrt{n}}\,d_{M^n})$. Therefore, we recover the well-known fact that $\text{Isom}(M^n)= \text{Isom}(M)^n \rtimes S_n$, which also follows directly from the de Rham decomposition theorem. A brief overview of standard probability spaces will be given in the preliminaries.

This isometric rigidity contrasts sharply with the flexibility of isometries in the classical $L^2(\Omega)=L^2(\Omega,\R)$. See Remark~\ref{R} for details.

We further observe a remarkable property: for atomless probability spaces $(\Omega,\mu)$, any factor in a direct product decomposition of $L^2(\Omega,M)$ is isometric to a rescaled version of the original space. In other words:
\begin{introthm} \label{L2factors1}
	Let $\Omega$ be a standard probability space without atoms and $M$ be a complete Riemannian manifold with irreducible universal cover and $\dim(M)\geq 2$. For any non-trivial direct product decomposition $L^2(\Omega,M)= Y\times \overline{Y}$, both factors $Y$ and $\overline{Y}$ are isometric, up to rescaling, to the original space $L^2(\Omega,M)$.
\end{introthm}

Thus the space $L^2(\Omega, M)$ lacks irreducible factors. This contrasts with the generalized de Rham decomposition theorem for metric spaces due to Lytchak and Foertsch (cf. \cite{FL}), which requires finite affine rank for unique decomposition into a Euclidean and irreducible factors. Our example demonstrates the necessity of this finiteness condition.

Finally we are able to show the following striking fact:
\begin{introthm} \label{MN}
	Let $\Omega$ be a standard probability space, and $M$ and $N$ be complete Riemannian manifolds with irreducible universal covering. Then $L^2(\Omega,M)$ is isometric to $L^2(\Omega,N)$ if and only if $M$ is isometric to $N$. 
\end{introthm}

This result, along with the isometric rigidity behind Theorem~\ref{semi-main}, parallels findings by Bertrand and Kloeckner for Wasserstein spaces (\cite{MR3509929}). There it was shown that, for a specific class of metric spaces, $\wass(X)$ and $\wass(Y)$ are isometric, if and only if $X$ and $Y$ are. This line of inquiry has recently been further explored in Che, Galaz-Garc\'ia, Kerin, and Santos-Rodr\'iguez \cite{che}.

For metric spaces $X$ and $Y$ we have a canonical isometry between $L^2(\Omega,X\times Y)$ and $L^2(\Omega,X)\times L^2(\Omega,Y)$. Moreover, for atomless $\Omega$, there exists an abstract one between $L^2(\Omega,X^n)$ and $L^2(\Omega,\sqrt{n}X)$. These facts illustrate that some irreducibility condition is necessary for the above results. For more details on this, we refer to section~\ref{nes}. 

We note that without the irreducibility assumption, we can still provide weaker algebraic characterizations of the isometry group. However, in contrast to Theorem~\ref{semi-main}, such characterizations do not explicitly reveal the structure of the isometries. We will not delve deeper into this in the present paper but as an illustration of such a case, we refer to Remark~\ref{rem} for a characterization of the isometry group of $L^2(\Omega, M)$ for simply connected $M$ and atomless $\Omega$.

\subsection{General strategy}

Let $X$ be a metric space and $\gamma_1:[0,a_1]\to X$, $\gamma_2:[0,a_2]\to X$ geodesic segments with $\gamma_1(0)=\gamma_2(0)=:x$. Recall that $\angle(\gamma_1,\gamma_2):=\limsup_{t,t^\prime \to 0}\overline{\angle}_{x}(\gamma_1(t),\gamma_2(t^\prime))$ denotes the \textit{Alexandrov angle} between $\gamma_1$ and $\gamma_2$, where $\overline{\angle}_{x}(\gamma_1(t),\gamma_2(t^\prime))$ is the Euclidean comparison angle at $x$. An Alexandrov angle \textit{exists in the strict sense} if the limit $\lim_{t,t^\prime \to 0}\overline{\angle}_{x}(\gamma_1(t),\gamma_2(t^\prime))$ exists. Let $X$ be a geodesic metric space. We say that \textit{angles exists in $X$} or \textit{X has angles} if the Alexandrov angle between any pair of geodesic segments issuing from the same point exists in the strict sense. This notion will play an important role in our reasoning and we will establish the following result of independent interest.

\begin{introlem}\label{angles}
	Let $X$ be a geodesic metric space in which angles exist and let $(\Omega,\mu)$ be a finite measure space. Then angles also exist in $L^2(\Omega,X)$.
\end{introlem}
We prove this result in section~\ref{winkel} as a natural extension of our preliminary discussion on angles.

We call a map $f:X \rightarrow Y$ between metric spaces \textit{affine} if it preserves the class of linearly reparametrized geodesics. Important examples are dilations and projections $X\times Y \to Y$. 

Let $M$ be a Riemannian manifold of dimension at least two with irreducible universal covering and $Y$ be a geodesic metric space in which angles exist.
It turns out that all affine maps $f:M\to Y$ are trivial: they rescale distances uniformly, i.e. there exists $c \geq 0$ such that for all $p,q\in M$, $d_Y(f(p),f(q))=c\,d_M(p,q)$. 
This is a special case of the following more general statement which we will prove in section~\ref{MY}: for an affine map $f:M_1\times ... \times M_n \to Y$, where each $M_i$ is as above, there exist $c_1,...,c_n \geq 0$ such that for all $(p_1,...,p_n), (q_1,...,q_n) \in M_1\times ... \times M_n$, \[d_Y^2(f(p_1,...,p_n),f(q_1,...,q_n))=\sum_{i=1}^{n}c_i^2 d_{M_i}^2(p_i,q_i).\]
In particular, this applies to affine maps $f:M^n\to Y$, and the key idea is to extend this result to the space $L^2(\Omega,M)$ that can be interpreted as a generalized, weighted \textit{infinite} product of copies of $M$: we will show in section~\ref{L2MY} that for any affine map $f:L^2(\Omega,M) \rightarrow Y$, there exists $\eta \in L^\infty(\Omega)$, $\eta\geq 0$ such that for any $p,q \in L^2(\Omega,M)$, $$d_Y^2(f(p),f(q))=\int_\Omega \eta(\omega) d_M^2(p(\omega),q(\omega))\,d\mu(\omega).$$ Therefore, we characterized all affine maps from $L^2(\Omega,M)$ into geodesic metric spaces with angles. 

These in particular encompass the projection mappings of any splitting $L^2(\Omega,M)=Y\times \overline Y$:
by Lemma~\ref{angles}, and since angles exist in the above sense in Riemannian manifolds (cf. Burago-Burago-Ivanov \cite{Burago}), we know that  $L^2(\Omega,M)$ has angles.
Thus given any splitting $L^2(\Omega,M)= Y\times \overline Y$, we deduce that $Y$ and $\overline Y$ have angles and hence that the projections $L^2(\Omega,M)\to Y$ and $L^2(\Omega,M)\to \overline Y$ are affine maps of the above type. 

We make use of this observation in section~\ref{split} to show that, for any splitting $L^2(\Omega,M)=Y\times \overline Y$, there exists a measurable $A\subset \Omega$ such that $Y$  and $\overline Y$ are canonically isometric to $L^2(A,M)$ and $L^2(A^c,M)$, respectively. For atomless $\Omega$, this will lead to Theorem~\ref{L2factors1}.

Further, in section~\ref{locrig}, we use this to show that isometries are localisable in the following sense: given any isometry $\gamma:L^2(\Omega,M) \rightarrow L^2(\Omega,N)$ and measurable $A\subset \Omega$, there exists $\Psi(A)\subset \Omega$, such that $$\int_A d_M^2(f(\omega),g(\omega))d\mu(\omega) = \int_{\Psi(A)} d_N^2(\gamma(f)(\omega),\gamma(g)(\omega))\,d\mu(\omega).$$ 

 We observe that $A\mapsto \mu(\Psi^{-1}(A))$ is a measure. Using this insight, we eventually recover $\varphi \in \text{Aut}(\Omega)$ and $\rho\in L^2(\Omega,\text{Isom}(M,N))$ such that for $\mu$-a.e. $\omega\in \Omega$, $\gamma(f)(\omega)=\rho(\varphi(\omega))(f(\varphi(\omega))).$ Indeed we prove:
\begin{introthm}\label{isometriesM}
	Let $(\Omega,\mu)$ be a standard probability space, and $M$ and $N$ be complete Riemannian manifolds of dimension at least two with irreducible universal covers. A map $\gamma:L^2(\Omega,M) \to L^2(\Omega,N)$ is an isometry if and only if there exist $\varphi\in \text{Aut}(\Omega)$ and $\rho\in L^2(\Omega,\text{Isom}(M,N))$, such that for $\mu$-a.e. $\omega\in \Omega$, $\gamma(f)(\omega)=\rho(\varphi(\omega))(f(\varphi(\omega))).$
\end{introthm}

Here $L^2(\Omega, \text{Isom}(M,N))$ denotes the set of maps $\rho: \Omega \to \text{Isom}(M,N)$ such that for some (and hence any) $x \in M$, the map $\omega \mapsto \rho(\omega)(x)$ lies in $L^2(\Omega, N)$.  The existence of such a map implies the existence of an isometry between $M$ and $N$, which implies Theorem~\ref{MN} in the case $\text{dim}(M), \text{dim}(N)\geq 2$.  Indeed both Theorems~\ref{semi-main} and \ref{MN} are corollaries of Theorem~\ref{isometriesM}, and complete proofs are provided in Section~\ref{final}.
 
We finally remark that we can analogously define spaces $L^p(\Omega, M)$ for $p\neq2$. A natural and interesting question is whether the above results still hold for these spaces. Even though many parts of our strategy naturally carry over, some key geometric arguments fail in this case. In particular, the existence of angles established in Lemma~\ref{angles} fails for general $L^p(\Omega, X)$. Given the reasons above and the geometric relevance of $L^2$-spaces mentioned at the beginning of the paper, we restrict our attention to the case $p=2$.\\

\section{Preliminaries} \label{sec_pre}
\subsection{Measure spaces}\label{measure}

We briefly introduce the measure-theoretic framework underlying our work. For more details, we refer to Bogachev \cite[Chapter 9]{MR2267655}. 
\begin{definition} Let $(\Omega_1,\mathcal A_1, \mu_1)$ and $(\Omega_2,\mathcal A_2,\mu_2)$ be measure spaces. 
	\begin{enumerate}
		\item A bijection $\varphi:\Omega_1 \to \Omega_2$ with $A\in \mathcal A_1 \iff \varphi(A) \in \mathcal A_2$ and $\mu_2(\varphi(A))=\mu_1(A)$ for all $A\in \mathcal A_1$ is a \textbf{strict isomorphism}.
		\item A map $\varphi:\Omega_1 \to \Omega_2$ for which there exist $N_1\in \mathcal A_1$ and $N_2 \in \mathcal A_2$ with $\mu_1(N_1)=\mu_2(N_2)=0$, and such that $\Omega_1\setminus N_1 \to \Omega_2 \setminus N_2$, $\omega \mapsto \varphi(\omega)$ is a strict isomorphism is called an \textbf{isomorphism}.
	\end{enumerate}
\end{definition} 

We also record, that similarly, $\varphi:\Omega_1 \to \Omega_2$ with $N_1\in \mathcal A_1$ and $N_2 \in \mathcal A_2$ such that $\mu_1(N_1)=\mu_2(N_2)=0$, and such that $\Omega_1\setminus N_1 \to \Omega_2 \setminus N_2$, $\omega \mapsto \varphi(\omega)$ is bijective is called \textit{almost everywhere bijective}.

A measure space isomorphism $\varphi:\Omega \to \Omega$ is an \textit{automorphism}. The group of automorphisms of $\Omega$ under composition is denoted Aut$(\Omega)$.

We typically omit the $\sigma$-algebra and write $(\Omega,\mu):=(\Omega,\mathcal{A},\mu)$.

\textit{Standard probability} or \textit{Lebesgue-Rokhlin spaces} were introduced by V.A. Rokhlin as probability spaces satisfying certain natural axiomatic properties. This class of spaces encompasses a very wide range of cases, including virtually all those arising naturally in geometric settings. Indeed, for example, the measure-theoretic completion of any Polish space equipped with its Borel $\sigma$-algebra and a probability measure yields a standard probability space (cf. Ito in \cite[Section 2.4]{Ito}).

The essential property for us will be that these spaces are always isomorphic to the Lebesgue unit interval and a collection of atoms: 

\begin{theorem}\label{Rokhlin}(Bogachev, \cite[Theorem 9.4.7]{MR2267655})
	Let $(\Omega,\mu)$ be a standard probability space. Then $\Omega$ is isomorphic to $[0,1]\sqcup \N$, equipped with the measure $\nu:=c\lambda+\sum_{n=1}^{\infty}p_n\delta_n$, where $c\geq 0$ and $p_n\geq 0$ for all $n\in \N$. \end{theorem}
	
Here $\lambda$ denotes the Lebesgue measure on $[0,1]$ and $\delta_{n}$ the Dirac measure at  $n \in \N$.  The atoms are represented by the integers and the space is atomless if $p_n=0$ for all $n\in \N$. Of course $c+\sum_{n=0}^\infty p_n=1$.
\subsection{Geodesics and affine maps}

Given an interval $I\subset \R$ and a metric space $(X,d)$, an isometric embedding $\gamma: I \rightarrow X$ is called a \textit{geodesic}. $(X,d)$ is called a \textit{geodesic metric space} if every two points in $X$ are joined by a geodesic. A map $f:X \to Y$ between metric spaces $X$ and $Y$ is a called a \textit{$\lambda$-dilation} if for all $x,y\in X$, $d_Y(f(x),f(y))=\lambda d_X(x,y)$. The set of $\lambda$-dilations $f:X \to Y$ is denoted $\text{Dil}_\lambda(X,Y)$; the set of all dilations, $\bigcup_{\lambda\geq 0}\text{Dil}_\lambda(X,Y)$, as $\text{Dil}(X,Y)$. A bijective $\lambda$-dilation with $\lambda=1$ is an isometry.
\begin{definition} 
	A map $f:X \rightarrow Y$ between metric spaces $X$ and $Y$ is called affine if  it preserves the class of linearly reparametrized geodesics. 
\end{definition}

In other words, given a geodesic $\gamma: I \rightarrow X$; $f(\gamma):I \rightarrow Y$, $t \mapsto f(\gamma(t))$ is a linearly reparametrized geodesic, i.e. there exists a constant $\rho(\gamma)\geq 0$, called reparametrization factor, such that for all $t,t^\prime \in I$, $$d_Y(f(\gamma(t)),f(\gamma(t^\prime)))= \rho(\gamma) |t-t^\prime|.$$
We stress that in contrast to a dilation, this factor $\rho:=\rho(\gamma)$ depends on the considered geodesic $\gamma$.

The most important natural examples of affine maps are dilations and projections $X\times Y \to X$. Indeed projections are affine since a map into a product space is a linearly reparametrized geodesic precisely when its components are (see Bridson and Haefliger \cite[I.5.3]{BH}).

Finally, for a metric space $X:=(X,d)$ and $a\geq 0$, $aX$ denotes the rescaled space $(X,a\cdot d)$. For $a=0$ this space collapses to a point. 
\subsection{$L^p$-function spaces with metric space targets}\label{Lp}
We recall the notion of $L^p$-spaces with metric space targets. See, for example, Korevaar and Schoen \cite{KS} and Monod \cite{Monod} for more detailed expositions. We begin by defining the space $L^p(\Omega, X)$ of $L^p$-functions $f:\Omega \longrightarrow X$.
\begin{definition}\label{def: L^2}
Let $(\Omega, \mu)$ be a finite measure space, $(X,d)$ be a metric space equipped with its Borel $\sigma$-algebra and $1\leq p < \infty$. We denote by $L^p(\Omega,X)$ the space of all measurable functions $f:\Omega\longrightarrow X$ (identified up to null-sets) having separable range and which satisfy, for some (and hence any) $x\in X$, $$\int_\Omega d^p(f(\omega),x)\,d\mu(\omega)<\infty.$$
\end{definition}
Note that if $f,f^\prime:\Omega \longrightarrow X$ are measurable, $(f,f^\prime):\Omega \longrightarrow X\times X$,  $\omega\mapsto (f(\omega),f^\prime(\omega))$ is also measurable. Thus in particular the function $d^p(f,f^\prime):\Omega \longrightarrow \R$, $\omega\mapsto d^p(f(\omega), f^\prime(\omega))$ is measurable and hence as a special case, the above integral is well defined. In light of this remark, the following is also well-defined: 
\begin{lemma}
	 Let $d_{L^p}: L^p(\Omega,X) \times L^p(\Omega,X) \longrightarrow \R_{\geq0}$ be given by: $$d_{L^p}(f,f^\prime):=\left(\int_\Omega d^p(f(\omega),f^\prime(\omega))\,d\mu(\omega)\right)^{\frac{1}{p}}.$$
	 Then $d_{L^p}$ defines a metric on $L^p(\Omega,X)$, the so-called $L^p$-metric.
\end{lemma}

 Using the change of variable formula, we obtain:
\begin{lemma}\label{lem:cv}
	Let $(X,d)$ be a metric space. If $(\Omega_1, \mu_1)$ and $(\Omega_2, \mu_2)$ are two isomorphic measure spaces with isomorphism $\varphi:\Omega_1\longrightarrow \Omega_2$, then $$\Psi: L^p(\Omega_2,X) \longrightarrow L^p(\Omega_1,X)$$ $$f\mapsto f\circ \varphi$$ is a well-defined isometry of metric spaces. 
\end{lemma}

Put in words, the isometry class of the space $L^p(\Omega,X)$, for fixed $(X,d)$, only depends on the measure space class of $\Omega$.

Next we record the following result characterizing the geodesics in $L^2$-spaces, see Monod \cite{Monod} for more details.

\begin{theorem}(Monod, \cite[Proposition 44]{Monod}) \label{thm: monodgeodesics}
	Let $I\subset \R$ be any interval. A continuous map $\sigma:I\longrightarrow L^2(\Omega,X)$ is a geodesic if and only if there is a measurable map $\alpha:\Omega \longrightarrow \R_{\geq 0}$ with $$\int_\Omega \alpha(\omega)^2 \,d\mu(\omega)=1$$ and a collection $\{\sigma^\omega\}_{\omega\in \Omega}$ of geodesics $\sigma^\omega:\alpha(\omega)I \longrightarrow X$ such that $$\sigma(t)(\omega)=\sigma^\omega(\alpha(\omega)t) \text{ for all } t\in I \text{ and } \mu\text{-almost every } \omega\in \Omega.$$ \end{theorem}

Here $\alpha(\omega)I$ denotes the interval obtained by scaling $I$ with $\alpha(\omega)$.

Further, we record the following, see again Monod and Korevaar-Schoen \cite{Monod, KS} for details. 
\begin{lemma}\label{lem: L2}
	Let $(\Omega, \mu)$ be a finite measure space. Then we have:\begin{enumerate}
		\item If $(X,d)$ is a complete metric space, then so is $(L^2(\Omega,X),d_{L^2})$.
		\item If $(X,d)$ is a (uniquely) geodesic metric space, then so is $(L^2(\Omega,X),d_{L^2})$. 
	\end{enumerate}
\end{lemma}

Let $\text{FP}(\Omega)=\bigcup_{n\in \N}\{(A_i)_{i=1}^n\subset \mathcal A^n, \bigsqcup_{i=1}^n A_i=\Omega\}$ be the collection of measurable finite partitions of $\Omega$. 

Let $(\Omega,\mu)$ be a finite measure space and $X$ be a metric space. For a partition $\alpha=(A_1,..,A_n)\in \text{FP}(\Omega)$ and $x=(x_1,...,x_n)\in X^n$, we define the simple function $f_x^\alpha:\Omega \to X$ by $f_x^\alpha(\omega)=x_i$ for $\omega\in A_i$. The map $X^n \to L^2(\Omega,X)$, given by $x\mapsto f_x^\alpha,$ is an isometric embedding after factor-wise rescaling: \[d_{L^2}^2(f_x^\alpha,f_{x^\prime}^\alpha)=\sum_{i=1}^n \mu(A_i)d^2_X(x_i,x_i^\prime).\]
In particular, this map is affine. The image of this map will be denoted as $C(\alpha)$. We now record the density of simple functions in $L^2(\Omega,X)$. \begin{lemma}\label{simple}
The simple functions $\bigcup_{\alpha\in \text{FP}(\Omega)}C(\alpha)$ are dense in $(L^2(\Omega,X),d_{L^2})$.
\end{lemma}

This result is standard, but a precise reference for the case of metric space targets is difficult to locate. For completeness, we include a proof.

\begin{proof}[Proof of Lemma~\ref{simple}]
	
First, assume $f \in L^p(\Omega, X)$ is bounded, i.e., there exists $M > 0$ such that $d(f(\omega), x_0) \leq M$ for all $\omega \in \Omega$, where $x_0$ is some fixed element in $X$. Since $f$ has separable range, its image $f(\Omega)$ is a separable subset of $X$. Let $D = \{x_k\}_{k=1}^\infty$ be a countable dense subset of $f(\Omega)$. Since $\mu$ is a finite measure, $\mu(\Omega) < \infty$.

For each $n \in \mathbb{N}$, since $D$ is dense in $f(\Omega)$, for each $\omega \in \Omega$, there exists $x_k \in D$ such that $d(f(\omega), x_k) < \frac{1}{n}$. Define $E_{n,k} = \{\omega \in \Omega : d(f(\omega), x_k) < \frac{1}{n} \text{ and } d(f(\omega), x_j) \geq \frac{1}{n} \text{ for all } j < k \}$. The sets $E_{n,k}$ are measurable, disjoint, and their union is $\Omega$. Let $\epsilon > 0$. Choose $N_n$ such that $\mu(\bigcup_{k=N_n+1}^\infty E_{n,k}) < \frac{\epsilon^p}{2M^p}$. Let $\alpha_n := (E_{n,1}, \dots, E_{n,N_n}, \bigcup_{k=N_n+1}^\infty E_{n,k}) \in \text{FP}(\Omega)$ and $s_n := f_x^{\alpha_n}\in C(\alpha_n)$, where $x = (x_1, \dots, x_{N_n}, x_0)$ with $x_i \in D$. If $\omega \in E_{n,k}$ for some $k \leq N_n$, then $d(f(\omega), s_n(\omega)) < \frac{1}{n}$. If $\omega \in \bigcup_{k=N_n+1}^\infty E_{n,k}$, then $d(f(\omega), s_n(\omega)) = d(f(\omega), x_0)$.
Therefore,
\begin{align*} d_{L^p}^p(f, s_n) &= \int_\Omega d(f(\omega), s_n(\omega))^p d\mu(\omega) \\ &\leq \frac{\mu(\Omega)}{n^p} + \int_{\bigcup_{k=N_n+1}^\infty E_{n,k}} d(f(\omega), x_0)^p d\mu(\omega) \\ &\leq \frac{\mu(\Omega)}{n^p} + M^p \mu(\bigcup_{k=N_n+1}^\infty E_{n,k}) < \frac{\mu(\Omega)}{n^p} + \frac{\epsilon^p}{2} \end{align*}
Choose $n$ large enough so that $\frac{\mu(\Omega)}{n^p} < \frac{\epsilon^p}{2}$. Then $d_{L^p}(f, s_n) < \epsilon$.

Now, let $f \in L^p(\Omega, X)$ be unbounded. For each $M > 0$, define $f_M(\omega) = f(\omega)$ if $d(f(\omega), x_0) \leq M$, and $f_M(\omega) = x_0$ otherwise. Then $f_M$ is bounded. By the previous argument, there exists a simple function $s_M \in C(\alpha)$ for some $\alpha$ such that $d_{L^p}(f_M, s_M) < \frac{\epsilon}{2}$.

Also, $d_{L^p}^p(f, f_M) = \int_{\{\omega: d(f(\omega), x_0) > M\}} d(f(\omega), x_0)^p d\mu(\omega)$. Since $f \in L^p(\Omega, X)$, we can choose $M$ large enough such that $d_{L^p}(f, f_M) < \frac{\epsilon}{2}$.

By the triangle inequality, $d_{L^p}(f, s_M) \leq d_{L^p}(f, f_M) + d_{L^p}(f_M, s_M) < \frac{\epsilon}{2} + \frac{\epsilon}{2} = \epsilon$. Thus, $\bigcup_{\alpha\in \text{FP}(\Omega)}C(\alpha)$ is dense in $(L^p(\Omega,X), d_{L^p})$. \end{proof}

\subsection{Angles in $L^2$-spaces}\label{winkel}

We recall some definitions and results around the concept of angles in metric spaces and prove Theorem~\ref{angles} from the introduction. See Bridson and Haefliger \cite[Chapter I.1]{BH} for more on comparison triangles and angles.

\begin{definition}(Alexandrov angle)
Let $X$ be a metric space and $\gamma_1:[0,a_1]\to X$, $\gamma_2:[0,a_2]\to X$ be geodesics such that $\gamma_1(0)=\gamma_2(0)=:x$. The \textbf{Alexandrov angle} between $\gamma_1$ and $\gamma_2$ is defined as $$\angle(\gamma_1,\gamma_2):=\limsup_{t,t^\prime \to 0}\overline{\angle}_{x}(\gamma_1(t),\gamma_2(t^\prime)),$$ where $\overline{\angle}_{x}(\gamma_1(t),\gamma_2(t^\prime))$ is the comparison angle at $x$ in the comparison triangle $\overline\Delta(x,\gamma_1(t),\gamma_2(t^\prime))\subset \mathbb E^2$.
\end{definition}

If the limit $\lim_{t,t^\prime \to 0}\overline{\angle}_{x}(\gamma_1(t),\gamma_2(t^\prime))$ exists, we say that this angle exists in the strict sense. A geodesic metric space $X$ is said to \textit{have angles} or \textit{angles exist in $X$} if the Alexandrov angle between any pair of geodesic segments issuing from the same point exists in the strict sense.

We record a slightly finer version of the fact that a normed space is Euclidean if and only if it has angles (cf. Bridson and Haefliger \cite[Proposition I.4.5]{BH}).
\begin{lemma}\label{semi-norm_angles}
	Let $(V,|\cdot|)$ be a semi-normed real vector space. The quotient space $X$ by the equivalence relation induced by $|\cdot|$ has angles if and only if $|\cdot|$ is induced by a semi-definite, symmetric bilinear form $h$ on $V$. It is an inner product if and only if $|\cdot|$ is a norm. 
	\end{lemma}

Before proving Lemma~\ref{angles}, we prove the following simpler fact for which we could not find a direct reference.

\begin{lemma}\label{angles_product}
	Let $X_1$ and $X_2$ be a geodesic metric spaces. Then $X_1\times X_2$ has angles if and only if both $X_1$ and $X_2$ have angles. 
\end{lemma}

\begin{proof} First, suppose that $X_1 \times X_2$ has angles. Since $X_1$ can be isometrically embedded into $X_1 \times X_2$ (via the map $x \mapsto (x, x_0)$ for some fixed $x_0 \in X_2$), it follows that if $X_1 \times X_2$ has angles, then $X_1$ must also have angles.  The same argument applies to $X_2$.

Conversely, suppose that $X_1$ and $X_2$ have angles. Let $\gamma, \eta: [0,a] \to X_1 \times X_2$ be geodesics emanating from a common point $(x_1, x_2)$. By Bridson and Haefliger \cite[I.5.3]{BH}, we can write $\gamma(t) = (\gamma_1(\alpha_1^1 t), \gamma_2(\alpha_1^2 t))$ and $\eta(t) = (\eta_1(\alpha_2^1 t), \eta_2(\alpha_2^2 t))$ for geodesics $\gamma_i, \eta_i$ in $X_i$ emanating from $x_i$, where $(\alpha_1^1)^2 + (\alpha_1^2)^2 = (\alpha_2^1)^2 + (\alpha_2^2)^2 = 1$.  Let $t, s \in (0, a]$. Then
\begin{align*}
\cos(\angle_{(x_1,x_2)}(\gamma(t),\eta(s))) &= \frac{t^2+s^2-d^2_{X_1\times X_2}(\gamma(t),\eta(s))}{2ts} \\
&= \sum_{i=1}^2 \alpha_1^i\alpha_2^i \frac{(\alpha_1^i t)^2+(\alpha_2^i s)^2-d^2_{X_i}(\gamma_i(\alpha_1^i t),\eta_i(\alpha_2^i s))}{2(\alpha_1^i t)(\alpha_2^i s)} \\
&= \sum_{i=1}^2 \alpha_1^i\alpha_2^i \cos(\angle_{x_i}(\gamma_i(\alpha_1^i t),\eta_i(\alpha_2^i s))).
\end{align*}
Since $X_i$ have angles, the limits $\lim_{t,s \to 0}\angle_{x_i}(\gamma_i(\alpha_1^i t),\eta_i(\alpha_2^i s))$ exist for $i=1,2$. Therefore, by the continuity of the cosine function, the limit
$
\lim_{t,s \to 0} \cos(\angle_{(x_1,x_2)}(\gamma(t),\eta(s)))$
exists. Since $\arccos$ is continuous on $[0,1]$, we conclude that $\lim_{t,s \to 0} \angle_{(x_1,x_2)}(\gamma(t),\eta(s))$ exists.  Thus, $X_1 \times X_2$ has angles.
\end{proof}

The proof of Theorem~\ref{angles} follows similarly, with the added consideration of interchanging a limit and integration by dominated convergence.
  \begin{proof}[Proof of Lemma~\ref{angles}]
	Let $\sigma_1,\sigma_2:[0,a]\to L^2(\Omega,X)$ be two geodesics issuing from a common point $f\in L^2(\Omega,X)$. By Theorem~\ref{thm: monodgeodesics}, there exist measurable maps $\alpha_i:\Omega \longrightarrow \R_{\geq 0}$ with $\int_{\Omega}\alpha_i(\omega)^2 \,d\mu(\omega)=1$ and collections $\{\sigma_i^\omega\}_{\omega \in \Omega}$ of geodesics $\sigma_i^\omega:[0,\alpha_i(\omega)\cdot a]\to X$, $i=1,2$ such that $\sigma_i(t)(\omega)=\sigma_i^{\omega}(\alpha_i(\omega)t)$ for all $t\in [0,a]$ and $\mu$-almost every $\omega \in \Omega$. By the assumptions, we know that the following limits exist, \[ \lim_{t,s \to 0} \cos(\overline{\angle}_{f(\omega)}(\sigma_1^\omega(t),\sigma_2^\omega(s)))=\lim_{t,s \to 0}\frac{t^2+{s}^2-d_X^2(\sigma_1^\omega(t),\sigma_2^\omega(s))}{2st}.\] 
On the other hand, by the definition of $d_{L^2}$ and $\|\alpha_1\|_{L^2}=\|\alpha_2\|_{L^2}=1$,  
\[
	\cos(\overline{\angle}_{f}(\sigma_1(t),\sigma_2(s)))=\frac{t^2+s^2-d_{L^2}^2(\sigma_1(t),\sigma_2(s))}{2st}=\int_\Omega H(s,t,\omega) \,d\mu(\omega), \] where $H(s,t,\omega):=\alpha_1(\omega)\alpha_2(\omega) \frac{(t\alpha_1(\omega))^2+(s\alpha_2(\omega))^2-d_X^2(\sigma_1^\omega(\alpha_1(\omega)t),\sigma_2^\omega(\alpha_2(\omega)s)}{2t\alpha_1(\omega)\cdot s\alpha_2(\omega)}$
is bounded above by $\alpha_1(\omega)\alpha_2(\omega)$, which is integrable. Thus, by dominated convergence, the limit $\lim_{t,s\to 0}\cos(\overline{\angle}_{f}(\sigma_1(t),\sigma_2(s)))$ exists, and $$\lim_{t,s\to 0}\cos(\overline{\angle}_{f}(\sigma_1(t),\sigma_2(s)))=\int_{\Omega}\alpha_1(\omega)\alpha_2(\omega)\cos(\angle_{f(\omega)}(\sigma^\omega_1,\sigma^\omega_2))\,d\mu(\omega).$$ \end{proof}

\section{Affine maps on certain products}\label{MY}
	
We characterize affine maps on products of Riemannian manifolds of dimension at least two with irreducible universal coverings. 

\begin{theorem}\label{affineprod}
	Let $(M_1, g_1), ..., (M_n,g_n)$ be complete Riemannian manifolds of dimension at least two with irreducible universal covers. For a geodesic metric space $Y$ with angles, every affine map $$f: M_1 \times ... \times M_n \to Y$$ is, up to factor-wise rescaling with some $c_i \geq 0$, an isometry $(M_1, c_1 g_1) \times ... \times (M_n, c_n g_n) \to f(M_1 \times ... \times M_n)$.
\end{theorem}

	We stress that, by allowing $c_i=0$, we allow for the case in which $(M_i,c_ig_i)$ collapses to a point.

\begin{proof}
	Let $p=(p_1,...,p_n)\in M_1\times ... \times M_n$ and $H_i:=Hol_{p_i}(M_i)$ be the holonomy group of $M_i$ at $p_i$. By the Berger-Simons theorem (cf. Simons, \cite{MR0148010}), either $H_i$ acts transitively on the unit sphere in $T_{p_i}M_i=:V_i$, or $M_i$ is a locally symmetric space of higher-rank. Without loss of generality, assume the first $m$ factors are of the first type and the rest are of the second type. 

For the first type, choose a unit vector $e_i\in V_i$. Then $\Sigma_i:=\langle e_i \rangle$ intersects all $H_i$-orbits, and the stabilizer $W_i:=\text{Stab}_{H_i}(\Sigma_i)$ contains an element which restricted to $\Sigma_i$ is the reflection $e_i \mapsto -e_i$. Indeed, by transitivity, for any other unit vector $v \in V_i$, there exists $h \in H_i$ such that $h(e_i) = v$. Thus the $H_i$-orbit of $e_i$ is the entire unit sphere in $V_i$, and hence $\Sigma_i$ intersects all $H_i$-orbits. Furthermore, since $-e_i$ is also a unit vector, there exists $g \in H_i$ such that $g(e_i) = -e_i$. As $g(\Sigma_i)=\Sigma_i$, $g \in W_i = \text{Stab}_{H_i}(\Sigma_i)$, and $g$ restricted to $\Sigma_i$ is the reflection $e_i \mapsto -e_i$.

For the second type, choose an immersed, maximal flat submanifold $F_i$. Then $\Sigma_i:=T_{p_i}F_i$ intersects all $H_i$-orbits orthogonally, and the stabilizer $W_i:=\text{Stab}_{H_i}(\Sigma_i)$ acts irreducibly on $\Sigma_i$ (cf. Eschenburg, \cite[Sections 9,10]{E}). 

We define $V:=T_p(M_1\times ... \times M_n)\cong V_1\oplus ... \oplus V_n$ and $H:=Hol_p(M_1\times ... \times M_n)\cong H_1\times ... \times H_n$.  Since $ H_1 \times ... \times H_n$ acts component-wise on $V_1\oplus ... \oplus V_n$, $\Sigma:=\Sigma_1\oplus ... \oplus \Sigma_n$ intersects all $H$-orbits.

For $v\in T(M_1\times ...\times M_n)$, define $|v|^f:=\frac{d(f(\gamma_v(t)), f(\gamma_v(t^\prime)))}{|t-t^\prime|}$ for sufficiently small $t,t^\prime$, where $\gamma_v$ is the Riemannian geodesic with $\gamma_v^\prime(0)=v$. This is well-defined precisely because $f$ is affine, and $|v|^f$ is the reparametrization factor of the (locally metric) geodesic $\gamma_v$ under $f$. By Lytchak \cite[Theorem 1.5]{affineimages}, $|\cdot|^f$ is a continuous family of semi-norms invariant under parallel transport. 

To prove our claim, it suffices to show that $|\cdot|^f$ is induced by a metric of the form $c_1g_1+...+c_ng_n$ for some $c_1,...,c_n\geq 0$. 

Since $|\cdot|^f$ is invariant under parallel transport, it suffices to show that its restriction to $V$, denoted $|\cdot|^f|_V$, is induced by $c_1g_1+\cdots+c_ng_n|_{V\times V}$, and we know that $|\cdot|^f|_V$ is invariant under the action of $H$. Consequently, as $\Sigma:=\Sigma_1\oplus \cdots \oplus \Sigma_n$ intersects all $H$-orbits, it suffices to show that $|\cdot|^f|_\Sigma$ is induced by $c_1g_1+\cdots+c_ng_n|_{\Sigma \times \Sigma}$. 

Let $W:=W_1\times ...\times W_n \leq H$.
As a side note, we remark that by Lytchak \cite[4.4]{affineimages}, the restriction $q \mapsto q|_{\Sigma}$ is actually a bijection between $H$-invariant norms on $V$ and $W$-invariant norms on $\Sigma$. 

We define $F := \text{im}(\gamma_{e_1}) \times \cdots \times \text{im}(\gamma_{e_m}) \times F_{m+1} \times \cdots \times F_n,$. This is an immersed flat submanifold of $M_1 \times \cdots \times M_n$ with $T_pF \cong \Sigma$. 
Thus there exists a neighbourhood $U(0)\subset (\Sigma,g_1+...+g_n\vert_{\Sigma \times \Sigma})$ which embeds isometrically into $(M_1\times ... \times M_n, g_1+...+g_n)$. The image of $U(0)$ under this embedding is denoted $N(p)\subset M_1\times ... \times M_n$ and the image $f(N(p))\subset Y$ is isometric to $N(p)\subset (M_1\times ... \times M_n,|\cdot|^f)$. Since $|\cdot|^f$ is invariant under parallel transport, $f(N(p))\subset Y$ is isometric to $U(0)\subset (\Sigma,|\cdot|^f\vert_{\Sigma})$. Thus, by Lemma~\ref{semi-norm_angles}, \( |\cdot|^f|_\Sigma \) is induced by a semi-definite symmetric bilinear form 
\( h : \Sigma \times \Sigma \to \mathbb{R} \).

Since  $|\cdot|^f|_\Sigma$ is invariant under the action of $W$, the bilinear form $h$ is of course also invariant under the same action.  

By representation-theoretic arguments, $W$-invariant inner products on $\Sigma$ are linear combinations of $W_i$-invariant inner products on $\Sigma_i$. This extends to semi-definite symmetric bilinear forms like $h$. Nonetheless, in the following we will  provide an alternative, more direct argument:

Restricting $h$ to $\Sigma_i$, we have $W_i$-invariant semi-definite symmetric bilinear forms $h_i$ on $\Sigma_i$.

 For $i\leq m$, we see directly that $h_i=c_ig_i\vert_{\Sigma_i\times \Sigma_i}$ for some $c_i\geq 0$. 
 
 For $i>m$, we first notice that $\Sigma_i^0:=\{v\in \Sigma_i: |v|^f=0\}\subset \Sigma_i$ is a $W_i$-invariant subspace and therefore, by irreducibility, either $|v|^f\vert_{\Sigma_i}\equiv 0$ or $|v|^f\vert_{\Sigma_i}$ is a norm. In the former case, $h_i=0$. In the latter case, $h_i$ is a scalar product invariant under the irreducible action of $W_i\leq O(\Sigma_i,g_i\vert_{\Sigma_i\times  \Sigma_i})$ and thus by the lemma of Schur, there exists $c_i>0$ such that $h_i=c_ig_i\vert_{\Sigma_i \times \Sigma_i}$. 
 
 Thus for all $1\leq i \leq n$, there exists $c_i\geq 0$ such that $h_i=c_ig_i\vert_{\Sigma_i\times \Sigma_i}$. 
 
 Now it just remains to show that $h=h_1+...+h_n$. This is equivalent to showing that for $i\neq j$, $v\in \Sigma_i$ and  $w\in \Sigma_j$, we have that $h(v,w)=0$. 
 
 If $i\leq m$, recall that $W_i\leq W$ contains an element which restricted to $\Sigma_i$ is the reflection $e_i \mapsto -e_i$. Thus, by the $W$-invariance, we have that $h(v,w)=h(-v,w)=-h(v,w)$ and therefore $h(v,w)=0$. 
 
 If on the other hand, $i,j>m$, we focus on the linear map $f:\Sigma\to \R$ given by $f(u):=h(v,u)$. We observe that ker$(f) \cap \Sigma_j \leq \Sigma_j$ is a $W_j$-invariant subspace of $\Sigma_j$. 
 Indeed to see this, simply note that for $w=(\operatorname{id},...,w_j,...,\operatorname{id})\in W_j\leq W$ and $u\in \text{ker}(f)\cap \Sigma_j$, we have that $f(w \cdot u)=h(v,w \cdot u)=h(w \cdot v, w \cdot u)=h(v,u)=0$. 
 Since $im(f)\neq \{0\}$, $dim(\text{ker}(f))=n-1$ and since also $dim(\Sigma_i)\geq 2$, we deduce that $dim(\text{ker}(f)\cap \Sigma_j)\geq 1$. Since the action of $W_j$ on $\Sigma_j$ is irreducible, we therefore have that $ker(f)\cap \Sigma_j=\Sigma_j$, i.e. $h(v,w)=0$. 
\end{proof}			
			
In other words, we have shown that under the given assumptions, for an affine map $f:M_1\times ... \times M_n \to Y$, there exist constants $c_1,...,c_n\geq 0$ such that for any $(p_1,...,p_n),(q_1,...,q_n)\in M_1\times ... \times M_n$, \[d_Y^2(f(p_1,...,p_n),f(q_1,...,q_n))=\sum_{i=0}^n c_i^2 d_{M_i}^2(p_i,q_i).\]  Note that here $d_{M_i}$ denotes the Riemannian distance on $M_i$.

In particular, this holds for affine maps $f:M^n\to Y$ and the objective of the next part is to extend this to $L^2(\Omega,M)$. 

For $n=1$, we obtain the following special case of the above result. In this situation, the claim also trivially holds for $\dim(M)=1$.

\begin{corollary}
		Let $(M, g)$ be a complete Riemannian manifold with irreducible universal covering. For a geodesic metric space $Y$ with angles, every affine map $f: M \to Y$ is, up to a rescaling by some $c \geq 0$, an isometry $(M,cg) \to f(M)$.
\end{corollary}

\section{Affine maps of $L^2(\Omega,M)$}\label{L2MY} 

Let $(\Omega,\mu)$ be a finite measure space and $X$ be a metric space. For $\eta \in L^\infty(\Omega)$, $\eta \geq 0$, we define a pseudo-metric on $L^2(\Omega,X)$ by 

\[d_{\eta}(f,g):=\left(\int_\Omega \eta\, d^2(f,g) \, d\mu \right)^{\frac{1}{2}}.\] 

We denote the induced quotient metric space by $(L_{\eta}^2(\Omega,X), d_{\eta})$. The projection map $p_\eta: L^2(\Omega,X) \to L_{\eta}^2(\Omega,X)$ is affine. Indeed for a geodesic $\sigma:I \to L^2(\Omega,X)$, $p_\eta \circ \sigma$ is a linearly reparametrized geodesic with reparametrization factor $\left(\int_\Omega \eta(\omega)\alpha^2(\omega)d\mu(\omega) \right)^{\frac{1}{2}}<\infty$, where $\alpha:\Omega \to \R_{\geq 0}$ with $\int_\Omega \alpha^2 \,d\mu=1$ is the measurable map corresponding to $\sigma$ via Theorem~\ref{thm: monodgeodesics}.   
\begin{theorem}\label{affine_classical}
	Let $M$ be a complete Riemannian manifold of dimension at least two with irreducible universal cover, and let $Y$ be a geodesic metric space with angles. A Lipschitz map $F:L^2(\Omega,M)\to Y$ is affine if and only if there exists a nonnegative $\eta\in L^\infty(\Omega)$ such that, for all $f,g\in L^2(\Omega,M)$,
\[d_Y^2(F(f),F(g))=\int_\Omega\eta \,d_M^2(f,g)\, d\mu.\]\end{theorem}

In other words, the theorem asserts that the image of the affine map\,\,\, $F:L^2(\Omega,M)\to Y$ is isometric to $L^2_\eta(\Omega,M)$ for some $\eta \in L^\infty(\Omega)$, $\eta\geq 0$. 

\begin{proof}[Proof of Theorem~\ref{affine_classical}] Assume $F:L^2(\Omega,M)\longrightarrow Y$ is Lipschitz, and that there exists a nonnegative $\eta\in L^\infty(\Omega)$ such that, for all $f,g\in L^2(\Omega,M)$,
\[d_Y^2(F(f),F(g))=\int_\Omega\eta \,d_M^2(f,g)\, d\mu.\]
Let $\sigma: I \longrightarrow L^2(\Omega,M)$ be a geodesic. By Theorem~\ref{thm: monodgeodesics}, there exist a measurable $\alpha:\Omega \longrightarrow \R_{\geq 0}$ with $\int_\Omega \alpha^2 d\mu = 1$, and a family of geodesics $\{\sigma^\omega\}_{\omega \in \Omega}$, $\sigma^\omega:\alpha(\omega)I \longrightarrow X$, such that $\sigma(t)(\omega)=\sigma^\omega(\alpha(\omega)t)$ for almost every $\omega \in \Omega$. Then
\begin{align*}
d_Y(F(\sigma(t)), F(\sigma(t^\prime))) &= \left(\int_\Omega \eta(\omega) d_X^2(\sigma^\omega(\alpha(\omega)t),\sigma^\omega(\alpha(\omega)t^\prime))d\mu(\omega) \right)^{\frac{1}{2}} \\
&= \left(\int_\Omega \eta(\omega) \alpha^2(\omega)d\mu(\omega) \right)^{\frac{1}{2}}|t-t^\prime|.
\end{align*}
Thus $F\circ \sigma$ is a linearly reparameterized geodesic with factor $\left(\int_\Omega \eta(\omega) \alpha^2(\omega)d\mu(\omega) \right)^{\frac{1}{2}}$, and therefore $F$ is affine. 

 For the other direction, let us assume that we are given an affine Lipschitz map $F:L^2(\Omega,M) \longrightarrow Y$.

As noted above in section~\ref{Lp}, for a finite partition $\alpha=(A_1,...,A_n)\in \text{FP}(\Omega)$, the map $M^n \to L^2(\Omega,M)$, given by $x\mapsto f_x^\alpha$, is affine. Thus we obtain an affine map 
$$M^n \xrightarrow{x\mapsto f_x^\alpha} L^2(\Omega,M) \xrightarrow{F} Y.$$ 
Therefore, by Theorem~\ref{affineprod}, there exist $\lambda_1^\alpha,...,\lambda_n^\alpha \geq 0$ such that
\begin{equation} \label{eq:1}
d_Y^2(F(f_x^\alpha),F(f_{x^\prime}^\alpha))=\sum_{i=1}^{n}(\lambda_i^\alpha)^2 d_M^2(x_i,x_i^\prime).
\end{equation}

Now, we fix a set $A \in \mathcal{A}$.  Let $\alpha = (A_1, ..., A_{i-1}, A, A_{i+1}, ..., A_n)$ and $\beta = (B_1, ..., B_{j-1}, A, B_{j+1}, ..., B_m)$ be two partitions containing the same set $A$.  Consider $p, p' \in M$ with $p \neq p'$. Let $x^l := (p, ..., p)\in M^l$ for any $l\in \N$ (all components are $p$) and define $x^l_k := (p, ..., p, p', p, ..., p)\in M^l$, where $p^\prime$ is in the $k$-th position, for all $l\in \N$ and $k\leq l$. 

Then $f_{x^n}^\alpha = f_{x^m}^\beta$ (constant function $p$).  $f_{x^n_i}^\alpha$ is $p^\prime$ on $A$ and $p$ on $\Omega \setminus A$. Similarly, $f_{x_j^m}^\beta$ is $p^\prime$ on $A$ and $p$ on $\Omega \setminus A$. Therefore, $f_{x_i^n}^\alpha = f_{x_j^m}^\beta$. Thus,
$$d_Y(F(f_{x^n}^\alpha), F(f_{x_i^n}^\alpha)) = d_Y(F(f_{x^m}^\beta), F(f_{x^m_j}^\beta)).$$

By equation \eqref{eq:1}, we have:
$$d_Y(F(f_{x^n}^\alpha), F(f_{x^n_i}^\alpha)) = \lambda_i^\alpha d_M(p, p')$$
and
$$d_Y(F(f_{x^m}^\beta), F(f_{x_j^m}^\beta)) = \lambda_j^\beta d_M(p, p').$$

Thus, $\lambda_i^\alpha = \lambda_j^\beta$. This shows that the coefficient associated with the set $A$ depends only on $A$ and not on the specific partition containing $A$ or its position within that partition.
Hence, we can define $\lambda^A := \lambda_i^{\alpha}$ for some partition $\alpha$ containing $A$ (where $A$ is the $i$-th element of $\alpha$).  We then define $\widetilde{\mu}:\mathcal A \to \R_{\geq 0}$ by $\widetilde{\mu}(A):={(\lambda^A)}^2$.

 By the aforementioned independence of the $\lambda_i$, we therefore have
\begin{align}\label{discrete}
	\begin{split}
		d_Y^2(F(f_x^\alpha),F(f_{x^\prime}^\alpha))=\sum_{i=1}^{n}\widetilde{\mu}(A_i)d_M^2(x_i,x_i^\prime).
	\end{split}
\end{align}
We note that $\widetilde{\mu}(\emptyset)=0$. Let $\|F\|\geq 0$ be the Lipschitz constant of $F$. Since for $A\in \mathcal A$ and distinct $p,p^\prime \in M$, \begin{align*}
	\begin{split}
		\widetilde{\mu}(A) d_M^2(p,p^\prime)&=d_Y^2(F(f_{(p,p)}^{(A,A^c)}),F(f_{(p^\prime,p)}^{(A,A^c)})\leq \Vert F \Vert^2 d_{L^2}(f_{(p,p)}^{(A,A^c)},f_{(p^\prime,p)}^{(A,A^c)}) \\&= \Vert F \Vert^2\mu(A)d_M^2(p,p^\prime),
	\end{split}
\end{align*}
we also established that for all $A\in \mathcal A$,
\begin{align}\label{upperbound}
	\begin{split}
		\widetilde{\mu}(A)\leq \Vert F \Vert^2 \mu(A).
	\end{split}
\end{align}
Next notice that for disjoint $A,B \in \mathcal A$, and distinct $p,p^\prime \in M$, \begin{align*}
	\begin{split}
		\widetilde\mu(A\cup B)d_M^2(p,p^\prime)&=d_Y^2(F(f_{(p,p)}^{(A\cup B,A^c\cap B^c)}),F(f_{(p^\prime,p)}^{(A\cup B,A^c\cap B^c)}))\\&=d_Y^2(F(f_{(p,p,p)}^{(A,B,A^c\cap B^c)}),F(f_{(p^\prime,p^\prime,p)}^{(A,B,A^c\cap B^c)}))\\&=\widetilde\mu(A)d_M^2(p,p^\prime)+\widetilde\mu(B)d_M^2(p,p^\prime),
	\end{split}
\end{align*}
and so we deduce that $\widetilde\mu(A\cup B)=\widetilde\mu(A)+\widetilde\mu(B)$. By induction, $\widetilde{\mu}$ is finitely additive.  Now for disjoint $(A_i)_{i\in\mathbb{N}}$, finite additivity yields  $\widetilde{\mu}(\bigsqcup_{i=1}^\infty A_i) = \sum_{i=1}^n \widetilde{\mu}(A_i) + \widetilde{\mu}(\bigsqcup_{i=n+1}^\infty A_i)$. By (\ref{upperbound}) and the continuity of $\mu$, $\widetilde{\mu}(\bigsqcup_{i=n+1}^\infty A_i) \leq \|F\|^2 \mu(\bigsqcup_{i=n+1}^\infty A_i) \to 0$ as $n \to \infty$. Thus $\widetilde{\mu}$ is $\sigma$-additive, and since it is also non-negative and $\widetilde{\mu}(\emptyset) = 0$, it is a measure.

Also by the upper bound, $\widetilde\mu \ll \mu$. Thus by Radon-Nikodym, there exists $\eta \in L^1(\Omega)$ such that for all $A \in \mathcal A$, $\widetilde\mu(A)=\int_A \eta(\omega) \,d\mu(\omega)$. Again by (\ref{upperbound}), we deduce that $\Vert \eta \Vert_\infty \leq \Vert F \Vert^2$, and so $\eta \in L^\infty(\Omega)$ with $\eta \geq 0$.

Note, by (\ref{discrete}), for any $\alpha \in \text{FP}(\Omega)$ and $f, g \in C(\alpha)$, we have:
\begin{align*}
d_Y^2(F(f), F(g)) &= \int_\Omega d_M^2(f(\omega), g(\omega)) \tilde{d\mu}(\omega) \\
&= \int_\Omega \eta(\omega) d_M^2(f(\omega), g(\omega)) d\mu(\omega) = d_\eta^2(f, g).
\end{align*}
This establishes the equality $d_Y(F(f), F(g)) = d_\eta(f, g)$ for $f, g \in C(\alpha)$.

Now let $f, g \in \bigcup_{\alpha \in \text{FP}(\Omega)} C(\alpha)$.  This means there exist partitions $\alpha_1, \alpha_2 \in \text{FP}(\Omega)$ such that $f \in C(\alpha_1)$ and $g \in C(\alpha_2)$. Let $\alpha$ be the common refinement of $\alpha_1$ and $\alpha_2$. Then both $f$ and $g$ are in $C(\alpha)$.  This is because if $\alpha$ is a refinement of $\alpha_1$, then $C(\alpha_1) \subseteq C(\alpha)$. Thus, in fact, we already have $d_Y(F(f), F(g)) = d_\eta(f, g)$ for $f, g \in \bigcup_{\alpha \in \text{FP}(\Omega)} C(\alpha)$.

Therefore finally, by the density of $\bigcup_{\alpha \in \text{FP}(\Omega)} C(\alpha)$ in $L^2(\Omega,X)$ (see Lemma~\ref{simple}), the claim follows for all $f,g\in L^2(\Omega,X)$. 
 \end{proof}

\section{Main argument}
\subsection{Splittings of $L^2(\Omega,M)$}\label{split}
We analyse splittings of $L^2$-spaces of the above type and prove Theorem~\ref{L2factors1}.
\begin{corollary}\label{L2factors}
	Let $M$ be a Riemannian manifold of dimension at least two with irreducible universal cover, and let $L^2(\Omega,M) = Y \times \overline{Y}$. Then there exists a measurable $A \subseteq \Omega$ such that, for all $f,g \in L^2(\Omega,M)$, 

\[d_Y(P^Y(f),P^Y(g)) = \int_A d_M^2(f(\omega),g(\omega))\,d\mu(\omega)\]
and 
\[d_{\overline{Y}}^2(P^{\overline{Y}}(f),P^{\overline{Y}}(g)) = \int_{A^c} d_M^2(f(\omega),g(\omega))\,d\mu(\omega).\]
\end{corollary}

Thus we  show that $Y \to L^2(A,M) \text{, }P^Y(f) \mapsto f\vert_A$ and $\overline{Y} \to L^2(A^c,M)\text{, }P^{\overline{Y}}(f) \mapsto f\vert_{A^c},$ $f\in L^2(\Omega,M)$ are well-defined isometric embeddings. Since clearly every function in $L^2(A,M)$ is a restriction $f|_A$ of some $f$ in $L^2(\Omega,M)$, the maps are surjective and thus isometries.

\begin{proof}[Proof of Corollary~\ref{L2factors}]

Note that in Riemannian manifolds angles exist in the sense of section~\ref{winkel} (see for example Burago-Burago-Ivanov \cite{Burago}). Accordingly, by Lemma~\ref{angles}, since $M$ is a Riemannian manifold, $L^2(\Omega,M)=Y\times \overline Y$ admits angles. Thus by Lemma~\ref{angles_product}, angles exist in both $Y$ and $\overline{Y}$.

Therefore we may apply Theorem~\ref{affine_classical} to the affine projections onto the factors $Y$ and $\overline Y$ respectively, $P^Y:L^2(\Omega,M) \to Y$ and $P^{\overline Y}:L^2(\Omega,M) \to \overline Y$.  Thus there are functions $\eta, \overline \eta \in L^\infty(\Omega)$, $\eta, \overline \eta \geq 0$ such that for all $f,g\in L^2(\Omega,M)$, $d_Y(P^Y(f),P^Y(g))=d_\eta(f,g)$ and $d_{\overline Y}(P^{\overline Y}(f),P^{\overline Y}(g))=d_{\overline\eta}(f,g)$. Hence, we obtain a splitting \begin{align}\label{spiltting} L^2(\Omega,M)\xrightarrow{(p_\eta, p_{\overline \eta})}L^2_\eta(\Omega,M)\times L^2_{\overline\eta}(\Omega,M).\end{align} To establish our claim, all we need to show is that $\eta=\chi_A$ and $\overline \eta =\chi_{A^c}$ for some measurable $A\subset \Omega$.

This follows from the fact that (\ref{spiltting}) is a splitting:
we begin by observing that, for all $f,g\in L^2(\Omega,M)$, we know that 
\begin{align*}
	\int_\Omega d_M^2(f(\omega),g(\omega))\,d\mu(\omega)
&=d_{L^2}^2(f,g)=d_{\eta}^2(f,g)+d_{\overline{\eta}}^2(f,g)  \\ &= \int_\Omega (\eta(\omega)+\overline{\eta}(\omega))d_M^2(f(\omega),g(\omega))\,d\mu(\omega).
\end{align*}
Now we proceed in two steps. 
First, we show that for $\mu$-a.e. $\omega \in \Omega$, $\eta(\omega)+\overline\eta(\omega)=1$.
To that end, set $A:=(\eta +\overline\eta -1)^{-1}((-\infty,0))$ and choose distinct $p,q\in M$. By the above, we know that \begin{align*} 0=&\int_{\Omega} (\eta(\omega)+\overline\eta(\omega)-1)d_M^2(f_{(p,p)}^{(A,A^c)}(\omega),f_{(p,q)}^{(A,A^c)}(\omega))\,d\mu(\omega) \\=&\int_{A^c} (\eta(\omega)+\overline\eta(\omega)-1)d_M^2(p,q)\,d\mu(\omega). \end{align*}
Therefore, we deduce that $\eta(\omega)+\overline\eta(\omega)=1$ for $\mu$-a.e. $\omega\in A^c$. Analogously, we show that $\eta(\omega)+\overline\eta(\omega)=1$ for $\mu$-a.e. $\omega\in A$ and thus we have $\eta(\omega)+\overline\eta(\omega)=1$ for $\mu$-a.e. $\omega \in \Omega$. 

Secondly, we claim that in fact $\eta(\omega), \overline{\eta}(\omega)\in \{0,1\}$ for $\mu$-a.e. $\omega\in \Omega$. Indeed otherwise, there would exist a subset $A\subset \Omega$ with $\mu(A)>0$ and $\eta(\omega), \overline{\eta}(\omega) \in (0,1)$ for all $\omega\in A$. By the surjectivity of (\ref{spiltting}), for distinct $p,q\in M$, there would then exist $h\in L^2(\Omega,M)$ such that $(p_\eta(h),p_{\overline{\eta}}(h))=(p_\eta(f_{(p,p)}^{(A,A^c)}),p_{\overline{\eta}}(f_{(q,q)}^{(A,A^c)}))$. Thus $\int_\Omega \eta \,d_M^2(h,p)\,d\mu =0$ and therefore $h\vert_A\equiv p$, and likewise $h\vert_A\equiv q$. This is a contradiction and so indeed $\eta(\omega), \overline{\eta}(\omega)\in \{0,1\}$ for $\mu$-a.e. $\omega\in \Omega$. Thus taken together, there exists a measurable $A\subset \Omega$ such that $\eta=\chi_A$ and $\overline \eta =\chi_{A^c}$.

 This completes the proof.
\end{proof}

\begin{proof}[Proof of Theorem~\ref{L2factors1}] Since the standard probability space is atomless, by Theorem~\ref{Rokhlin}, we can assume that $\Omega=[0,1]$ equipped with the Lebesgue measure $\lambda$. 

Let $L^2(\Omega,M)=Y\times \overline Y$ be a nontrivial splitting. By Corollary~\ref{L2factors}, there exists a measurable $A\subset \Omega$ of positive measure such that $Y$ is isometric to $L^2(A,M)$. By normalizing the induced measure space on $A\subset \Omega$, we obtain another atomless standard probability space $(A,\frac{1}{\lambda(A)}\lambda)$ (cf. Bogachev \cite[Proposition 9.4.10]{MR2267655}). By Theorem~\ref{Rokhlin}, this space is again isomorphic to $([0,1],\lambda)$. 

Therefore, $Y\cong L^2(A,M)$ is isometric to $\lambda(A)L^2([0,1],M)$. 
\end{proof}
\subsection{Isometric localization and rigidity}\label{locrig} We show that the isometries of $L^2(\Omega,M)$ are localizable and prove Theorem~\ref{isometriesM}. In Remark~\ref{R} we describe how to obtain isometries of $L^2(\Omega,\R)$ violating the rigidity in the sense of Theorem~\ref{isometriesM}.

For the following result, recall the definition of an almost everywhere bijective map $\varphi:\Omega \to \Omega$ from section \ref{measure}.
\begin{lemma}\label{Psi}
	 Let $(\Omega,\mu)$ be a standard probability space and $M, N$ be complete Riemannian manifolds of dimension at least two with irreducible universal coverings. For an isometry $\gamma: L^2(\Omega,M) \to L^2(\Omega,N)$ there exists an almost everywhere bijective $\varphi:\Omega \to \Omega$ such that $\varphi_\ast \mu \ll \mu$ and $\mu \ll \varphi_\ast \mu$, and for every measurable $A\subset \Omega$ and $f,g\in L^2(\Omega,M)$,
$$\int_A d_M^2(f,g) \,d\mu = \int_{\varphi^{-1}(A)} d_N^2(\gamma(f),\gamma(g)) \,d\mu.$$
\end{lemma}

\begin{proof}

\begin{enumerate}[(i)]
\item 

We begin by pointing out that for any measurable $A\subset \Omega$, $$\int_\Omega d_M^2(f,g)\,d\mu=\int_Ad_M^2(f\vert_A,g\vert_A)\,d\mu+\int_{A^c}d_M^2(f\vert_{A^c},g\vert_{A^c})\,d\mu,$$ and thus $L^2(\Omega,M) \longrightarrow L^2(A,M)\times L^2(A^c,M)$, $f\mapsto (f \vert_A, f\vert_{A^c})$ is a canonical splitting of $L^2(\Omega,M)$. 
By combining this with $\gamma:L^2(\Omega,M) \to L^2(\Omega,N)$, we obtain a splitting of $L^2(\Omega,N)$:
$$L^2(\Omega,N) \xrightarrow[]{\gamma^{-1}} L^2(\Omega,M) \xrightarrow[]{f \mapsto (f\vert_A, f\vert_{A^c})} L^2(A,M)\times L^2(A^c,M).$$

By Corollary~\ref{L2factors}, there exists a measurable $\Psi(A) \subset \Omega$ such that $L^2(A,M) \to L^2(\Psi(A),N)$, $\gamma^{-1}(f)\vert_A \mapsto f\vert_{\Psi(A)}$ is an isometry. Thus for all $f,g \in L^2(\Omega,M)$,  $$\int_A d_M^2(f,g) \,d\mu=\int_{\Psi(A)} d_N^2(\gamma(f),\gamma(g))\,d\mu.$$

This uniquely defines $\Psi:\mathfrak{A} \longrightarrow \mathfrak{A}$, where $\mathfrak A$ denotes the quotient of the $\sigma$-algebra $\mathcal A$ of $\Omega$ by an identification up to null-sets.  

By considering the isometry $\gamma^{-1}:L^2(\Omega,N) \to L^2(\Omega,M)$, we analogously obtain $\Psi^\prime:\mathfrak{A} \longrightarrow \mathfrak{A}$ and $\Psi \circ \Psi^{\prime}=\Psi^{\prime} \circ \Psi =\operatorname{id}_\mathfrak{A}$. Thus $\Psi^{\prime}$ is an inverse and $\Psi$ is bijective.

Furthermore  $\Psi(\emptyset)=\emptyset$. Indeed, we know that for all $f, g\in L^2(\Omega,M)$,
$$\int_{\Psi(\emptyset)} d_N^2(\gamma(f),\gamma(g)) \,d\mu=\int_\emptyset d_M^2(f,g) \,d\mu = 0.$$

If $\mu(\Psi(\emptyset)) > 0$, we could choose $f,g$ such that $d_N^2(\gamma(f), \gamma(g)) > 0$ on $\Psi(\emptyset)$, leading to a contradiction after integration. Therefore, we must have $\mu(\Psi(\emptyset)) = 0$, and thus $\Psi(\emptyset) = \emptyset$ in the quotient algebra $\mathfrak{A}$.

Moreover, we know that for pairwise disjoint $(A_n)_{n \in \mathbb{N}} \subset \mathcal A$,
\begin{align*}
	\Psi\left(\bigsqcup_{n=1}^\infty A_n \right)=\bigsqcup_{n=1}^\infty \Psi(A_n).
\end{align*}
To see this, let $(A_n)_{n \in \mathbb{N}}$ be pairwise disjoint measurable subsets of $\Omega$ and $A := \bigsqcup_{n=1}^\infty A_n$.  We want to show that $\Psi(A) = \bigsqcup_{n=1}^\infty \Psi(A_n)$. First, let $f, g \in L^2(\Omega, M)$.  Then we know that
$\int_A d_M^2(f,g) \,d\mu = \sum_{n=1}^\infty \int_{A_n} d_M^2(f,g) \,d\mu$,
and by the definition of $\Psi$, we have that
$\int_{A_n} d_M^2(f,g) \,d\mu = \int_{\Psi(A_n)} d_N^2(\gamma(f),\gamma(g)) \,d\mu$.

Therefore,
\begin{align*}
	\int_A d_M^2(f,g) \,d\mu &= \sum_{n=1}^\infty \int_{\Psi(A_n)} d_N^2(\gamma(f),\gamma(g)) \,d\mu \\&= \int_{\bigsqcup_{n=1}^\infty \Psi(A_n)} d_N^2(\gamma(f),\gamma(g)) \,d\mu
\end{align*}

We also know that $\int_A d_M^2(f,g) \,d\mu = \int_{\Psi(A)} d_N^2(\gamma(f),\gamma(g)) \,d\mu$.

Comparing these two expressions, we get:
$$\int_{\Psi(A)} d_N^2(\gamma(f),\gamma(g)) \,d\mu = \int_{\bigsqcup_{n=1}^\infty \Psi(A_n)} d_N^2(\gamma(f),\gamma(g)) \,d\mu$$
Since this holds for all $f, g \in L^2(\Omega, M)$ and $\Psi$ is bijective, it implies that in the quotient algebra $\mathfrak{A}$, we have $\Psi(A) = \bigsqcup_{n=1}^\infty \Psi(A_n)$.

Applying the inverse to our result, we obtain \begin{align}\label{union}
	\Psi^{-1}\left(\bigsqcup_{i=1}^\infty A_i \right)=\bigsqcup_{i=1}^\infty \Psi^{-1}(A_i).
\end{align}
Thus $\widetilde\mu: \mathcal{A} \longrightarrow \R_{\geq 0}$, $A\mapsto \mu(\Psi^{-1}(A))$ defines a measure and since $\Psi$ is bijective and $\Psi(\emptyset)=\emptyset$, we have both $\widetilde\mu \ll \mu$ and $\mu \ll \widetilde\mu$. To see this, simply note that if $\mu(A)=0$ for some measurable $A\in \Omega$, then $A=\emptyset$ in $\mathfrak A$ and therefore $\Psi^{-1}(A)=\emptyset$ as well in $\mathfrak A$. Thus $\widetilde \mu (A)=\mu(\Psi^{-1}(A))=0$. The other claim follows analogously. 

The remainder of the proof relies on recovering a \textit{pointwise}, almost everywhere bijective $\varphi:\Omega \to \Omega$ which is such that $\varphi(\Psi(A))=A$ (up to null-sets) for all $A\in \mathcal A$. This finally establishes the claim since we then have both $\varphi_\ast \mu \ll \mu$, $\mu \ll \varphi_\ast \mu$, and of course by the above, $$\int_A d_M^2(f,g) \,d\mu = \int_{\varphi^{-1}(A)} d_N^2(\gamma(f),\gamma(g)) \,d\mu.$$
 
\item By Theorem~\ref{Rokhlin}, we can assume that $\Omega=[0,1]\sqcup \N$, equipped with the measure $c\lambda+\sum_{n=1}^{\infty}p_n\delta_n$, where $c\geq 0$, $p_n\geq 0$ for all $n\in \N$, and of course $c+\sum_{n=1}^\infty p_n=1$. Further, we can assume that if $p_i=0$, then $p_j=0$ for all $j\geq i$.

The bijection $\Psi$ respects the partition of $\Omega$ into an atomic and an atomless part: $\Psi$ permutes the atoms and thus preserves the set of Lebesgue measurable subsets, $\mathcal L([0,1])$, and the power set $\mathcal P(\N)$.
Indeed, otherwise, there would exist $\{i\}$ such that $\Psi(\{i\})$ is not an atom up to null-sets, and so it can be partitioned into two disjoint sets $A,B\in \mathcal A$ of positive measure. But then since $\Psi$ is a bijection, $\Psi^{-1}(A)$ and $\Psi^{-1}(B)$ are also disjoint sets of positive measure such that, by (\ref{union}), $\Psi^{-1}(A) \cup \Psi^{-1}(B)$ is a single point up to null-sets. This is a contradiction. Since this argument also holds for the inverse, $\Psi$ permutes the atoms. 
\item Next we recover the almost everywhere bijective $\varphi:\Omega \to \Omega$: by Radon-Nikodym, there exists $\rho \in L^1(\Omega)$ such that for all $A \in \mathcal{A}$, $\widetilde\mu(A)=\int_A \rho(\omega) \,d\mu(\omega)$ and $\rho>0$ $\mu$-almost everywhere.

We define $T:L^1(\Omega)\to L^1(\Omega)$ on simple functions by $T\left(\sum_i a_i\chi_{A_i}\right):=\rho\sum_i a_i\chi_{\Psi(A_i)}$, where $A_1,...,A_n\in \mathcal A$ are pairwise disjoint and $a_1,...,a_n \in \R$. This map is an isometry and it extends to a linear isometry on $L^1(\Omega)$ by density and a common refinement argument. 

Since $\Psi$ preserves $\mathcal L([0,1])$ and $\mathcal P(\N)$, $T$ decomposes into isometries $T_1:L^1([0,1],\lambda)\to L^1([0,1],\lambda)$ and $T_2:L^1(\N,\sum_{n\in \mathbb N} p_n\delta_n)\to L^1(\N,\sum_{n\in \mathbb N} p_n\delta_n)$. 

For the first isometry, we know that by Banach \cite[p.178]{Banach}, $T_1$ is given by $T_1(f)(\omega)=u_1(\omega)f(\varphi_1(\omega))$ for some measurable $u_1:[0,1]\to\mathbb{R}$ and a.e. bijective $\varphi_1:[0,1]\to[0,1]$.

For the second isometry $T_2$, let $J = \{n \in \mathbb{N} : p_n > 0\}$. Consider the isometry $\sigma: L^1(\mathbb{N}, \sum_{n \in \mathbb{N}} p_n \delta_n) \to L^1(J, \sum_{n \in J} \delta_n)$, defined  by $f \mapsto (p_n f(n))_{n \in J}$. This induces an isometry $\sigma \circ T_2 \circ \sigma^{-1}: L^1(J, \sum_{n \in J} \delta_n) \to L^1(J, \sum_{n \in J} \delta_n)$.
Under the assumptions stated at the beginning of part (ii), we have two cases. Either $J = \mathbb{N}$, in which case $\sigma \circ T_2 \circ \sigma^{-1}$ is an isometry of $L^1(\mathbb{N},\sum_{n\in \N}\delta_n)\cong\ell^1(\N)$. Therefore, again by Banach \cite[p. 178]{Banach}, $T_2$ must have the form $T_2(f)(\omega) = u_2(\omega) f(\varphi_2(\omega))$ for some permutation $\varphi_2: \mathbb{N} \to \mathbb{N}$ and a function $u_2: \mathbb{N} \to \R$.
Alternatively, $J = \{1, \dots, m\}$ for some $m \in \mathbb{N}$, and $\sigma \circ T_2 \circ \sigma^{-1}$ is an isometry of $L^1(\{1, \dots, m\},\sum_{n=1}^m \delta_n)$, which is isometric to $(\mathbb{R}^m, \|\cdot\|_1)$. The extreme points of the unit sphere in $(\mathbb{R}^m, \|\cdot\|_1)$ are precisely $\pm e_i$ for $i = 1, \dots, m$, where $e_i$ denotes the standard basis vector. Since linear isometries preserve extreme points, any linear isometry of $(\mathbb{R}^m, \|\cdot\|_1)$ must, up to a sign change for each coordinate, permute the coordinates. Therefore, also in this case, there exists a permutation $\varphi_2: \mathbb{N} \to \mathbb{N}$ and a function $u_2: \mathbb{N} \to \R$ such that $T_2(f)(\omega) = u_2(\omega) f(\varphi_2(\omega))$.

Combining $u_1$, $u_2$, $\varphi_1$, and $\varphi_2$, we obtain a measurable $u:\Omega\to\mathbb{R}$ and an a.e. bijective $\varphi:\Omega\to\Omega$ such that $T(f)(\omega)=u(\omega)f(\varphi(\omega))$ for all $f\in L^1(\Omega)$ and a.e. $\omega\in\Omega$.
\item Finally, we show that $\varphi(\Psi(A)) = A$, closing the argument. To do so, first note that $u > 0$ almost everywhere. For all measurable $A\subset \Omega$, we now know that $\rho\chi_{\Psi(A)}=T(\chi_A)=u\chi_A\circ \varphi$ and thus clearly $u\geq 0$ almost everywhere. Now assume there exists a set $A$ of positive measure on which $u$ vanishes. This would imply that for a.e. $\omega \in \Omega$, 
$\rho(\omega)\chi_{A}(\omega)=T(\chi_{\Psi^{-1}(A)})(\omega)=u(\omega)\chi_{\Psi^{-1}(A)}(\varphi(\omega)) = 0.$
Hence, $\rho$ would also vanish on $A$, contradicting the assumption that $\rho$ is strictly positive almost everywhere.
Therefore, for a.e. $\omega \in \Omega$, $\rho(\omega)\chi_{\Psi(A)}(\omega) = u(\omega)\chi_A(\varphi(\omega))$ and both $\rho, u > 0$. This forces $\varphi(\Psi(A)) = A$ up to null-sets. \end{enumerate} \end{proof}

\begin{proof}[Proof of Theorem~\ref{isometriesM}]
It is immediate that maps $\gamma:L^2(\Omega,M) \longrightarrow L^2(\Omega,N)$ of the given form are isometries.

Conversely, suppose $\gamma:L^2(\Omega,M) \longrightarrow L^2(\Omega,N)$ is an isometry. 
By Lemma~\ref{Psi}, there exists an almost everywhere bijective $\varphi:\Omega \to \Omega$ such that $\varphi_\ast \mu \ll \mu$, $\mu \ll \varphi_\ast \mu$, and for every  $A\in \mathcal A$ and $f,g\in L^2(\Omega,M)$, $$\int_A d_M^2(f,g) \,d\mu = \int_{\varphi^{-1}(A)} d_N^2(\gamma(f),\gamma(g)) \,d\mu.$$

By Radon-Nikodym, there exists a positive Radon-Nikodym derivative $\frac{d(\varphi_\ast\mu)}{d\mu}>0$. Thus by the change of variable formula for every $A\in \mathcal A$ and $f,g\in L^2(\Omega,M)$,
\begin{align*}
	\int_A d_M^2(f,g)\,d\mu 
	&= \int_{\varphi^{-1}(A)} d_N^2(\gamma(f),\gamma(g)) \,d\mu\\
	&= \int_{A} d_N^2(\gamma(f)\circ\varphi^{-1},\gamma(g)\circ\varphi^{-1})\,d(\varphi_\ast\mu) \\
	&= \int_A d_N^2(\gamma(f)\circ\varphi^{-1},\gamma(g)\circ\varphi^{-1})\,\frac{d(\varphi_\ast\mu)}{d\mu}\,d\mu.
\end{align*}
Thus, for $\mu$-a.e. $\omega\in\Omega$, 
\begin{equation}\label{eqn_dilation}
	d_M^2(f(\omega),g(\omega)) = \frac{d(\varphi_\ast\mu)}{d\mu}(\omega)\cdot d_N^2(\gamma(f)(\varphi^{-1}(\omega)),\gamma(g)(\varphi^{-1}(\omega))).
\end{equation}

 Our strategy now is twofold. First, we recover a family of dilations $\rho:\Omega \to \text{Dil}(M,N)$ such that, $\gamma(f)(\omega)=\rho(\varphi(\omega))(f(\varphi(\omega))$ for $\mu$-a.e. $\omega\in \Omega$, for all $f\in L^2(\Omega,M)$; secondly we exploit the fact that $M$ is a Riemannian manifold to show that in fact $\rho \in L^2(\Omega,\text{Isom}(M,N))$. 
 
To start with the former, define $\rho(\omega)(x):=\gamma(f^\Omega_x)(\varphi^{-1}(\omega))$. Then we have for all $x,y\in M$, \begin{equation}
	d_M^2(x,y) = \frac{d(\varphi_\ast\mu)}{d\mu}(\omega)\cdot d_N^2(\rho(\omega)(x),\rho(\omega)(y)).
\end{equation}

Since for $\mu$-a.e. $\omega\in \Omega$, $\frac{d(\varphi_\ast\mu)}{d\mu}(\omega)>0$, $\rho(\omega)$ is a dilation by the factor $\left(\sqrt{\frac{d(\varphi_\ast\mu)}{d\mu}(\omega)}\right)^{-1}>0$, and we obtain the family $\rho:\Omega \to \text{Dil}(M,N)$ up to null-sets. 

Furthermore observe that since (\ref{eqn_dilation}) holds $\mu$-a.e., the expression $\gamma(f)(\varphi^{-1}(\omega))$ only depends on the value $f(\omega)\in M$. Thus indeed $\gamma(f)(\omega)=\rho(\varphi(\omega))(f(\varphi(\omega)))$ for $\mu$-a.e. $\omega\in \Omega$.

For the second part, we start by recording that $\omega \mapsto \rho(\omega)(x)=\gamma(f^\Omega_x)(\varphi^{-1}(\omega))$ lies in $L^2(\Omega,N)$.

Next we note that there analogously exists a family of dilations $\rho^\prime:\Omega \to \text{Dil}(N,M)$ and an almost everywhere bijective $\varphi^\prime:\Omega \to \Omega$ such that, $\gamma^{-1}(f)(\omega)=\rho^\prime(\varphi^\prime(\omega))(f(\varphi^\prime(\omega)))$ for $\mu$-a.e. $\omega \in \Omega$ and all $f\in L^2(\Omega,N)$. But then, for $\mu$-a.e. $\omega \in \Omega$ we have that $\rho(\omega)\circ \rho^\prime(\varphi^\prime(\omega))=\operatorname{id} \vert_ N$. In other words, for $\mu$-a.e. $\omega \in \Omega$, $\rho(\omega)\in\text{Dil}(M,N)$ has a right inverse and is thus not just injective but bijective. 

Finally, to complete the proof we show that $\varphi$ is measure preserving, i.e. $\varphi\in \text{Aut}(\Omega)$ and therefore $\rho(\omega)\in \text{Isom}(M,N)$ for a.e. $\omega \in \Omega$ and thus in particular, $M$ is isometric to $N$.

To do so it suffices to show that $\frac{d(\varphi_\ast\mu)}{d\mu}$ is constant $\mu$-almost everywhere. Indeed in that case, if we let $c:=\frac{d(\varphi_\ast\mu)}{d\mu}$ and recall that $\varphi$ is almost everywhere bijective, we already know that $1=\mu(\Omega)=(\varphi_\ast\mu)(\Omega)=c\mu(\Omega)=c$, and therefore $\varphi \in \text{Aut}(\Omega)$.

To show that $\frac{d(\varphi_\ast\mu)}{d\mu}$ is constant $\mu$-almost everywhere we note that otherwise, there exist $\omega,\omega'\in \Omega$ with non-zero $\frac{d(\varphi_\ast\mu)}{d\mu}(\omega)\neq \frac{d(\varphi_\ast\mu)}{d\mu}(\omega')$, and such that $\rho(\omega),\rho(\omega')\in\text{Dil}(M,N)$ are bijective and well-defined. Therefore, we obtain a surjective dilation $\rho(\omega)^{-1}\circ\rho(\omega')$ with dilating factor $\lambda\neq 1$. But since $M$ is not Euclidean such dilations do not exist.

 Indeed $M$ admits a surjective dilation if and only if $M$ is Euclidean. Let $\alpha:M\to M$ be a surjective dilation with factor $\lambda\neq 1$. By taking the inverse if necessary we can assume that $\lambda<1$. Thus by Banach's fix point theorem, there exists $x\in M$ such that $\alpha(x)=x$. Now choose a sufficiently small compact ball $B\subset M$ around $x$. There exists a point $y\in B$ and some plane $\sigma \leq T_yM$ at which the sectional curvature $\kappa=\kappa(\sigma)$ is maximal in $B$. However, we also have sectional curvature $\frac{1}{\lambda^2}\kappa$ at $\alpha(y)\in B$ and $\frac{1}{\lambda^2}>1$. This implies that the sectional curvature vanishes throughout $B$. Thus we can isometrically embed $B$ into $\mathbb{E}^{\dim(M)}$. 
 
 Since $\alpha^{-1}$ scales $B$ by the factor $\lambda^{-1}>1$, we therefore have balls of arbitrary diameter around $x$ in $M$ which are Euclidean. Thus $M$ is itself Euclidean. 
This completes the proof.
\end{proof}

As a consequence, we obtain the following strengthened localization. 

\begin{corollary}
	 Let $M, N$ be complete Riemannian manifolds of dimension at least two with irreducible universal covers. For any isometry $\gamma: L^2(\Omega, M) \to L^2(\Omega, N)$ and measurable $A \subset \Omega$, there exists a measurable $B \subset \Omega$ with $\mu(A) = \mu(B)$ such that for all $f,g \in L^2(\Omega, M)$, 
$$\int_A d_M^2(f,g) \, d\mu = \int_B d_N^2(\gamma(f), \gamma(g)) \, d\mu.$$\end{corollary} 

\begin{remark}\label{R} Let $\Omega=[0,1]$.
	Consider the Hilbert space $L^2(\Omega)=L^2(\Omega,\R)$ with its standard inner product. The group of linear isometries of $L^2(\Omega)$ operates transitively on the unit sphere around $0\in L^2(\Omega)$. Now let $e:= \sqrt{2}\chi_{[0,1/2]}$ and $e^\prime:=\chi_{[0,1]}$. Evidently $e,e^\prime \in L^2(\Omega)$ lie on the unit sphere and thus there exists a linear isometry $T$ of $L^2(\Omega)$ such that $T(e)=e^\prime$.  
	
However, $T$ cannot be a rigid isometry in the sense of Theorem~\ref{isometriesM}. If it were, then for any measurable set $A \subseteq \Omega$, there would exist a set $B$ of equal measure such that for all $f \in L^2(\Omega)$,
$$\int_A |T(f)|^2 \,d\lambda= \int_B |f|^2\, d\lambda.$$

Applying this to $f=e$ and the set $A = [0, 1/2]$, we obtain a contradiction: 
$$1 = \int_A |e|^2 \, d\lambda= \int_{B} |T(e)|^2 \, d\lambda =\lambda(B)=\frac{1}{2}.$$
\end{remark}

\subsection{Final arguments}\label{final}
We provide proofs of Theorems~\ref{semi-main} and \ref{MN}. 
\begin{proof}[Proof of Theorem~\ref{semi-main}]
	
First, we show that $L^2(\Omega,\operatorname{Isom}(M))$ is normal in $\operatorname{Isom}(L^2(\Omega,M))$. For $\rho\in L^2(\Omega,\operatorname{Isom}(M))$, let $\gamma_\rho$ denote the isometry $f\mapsto \gamma_\rho(f)$, where $\gamma_\rho(f)(\omega) = \rho(\omega)(f(\omega))$. For $\varphi \in \text{Aut}(\Omega)$ on the other hand, let $\gamma^\varphi$ denote the isometry $f\mapsto f\circ \varphi$. Let $\tau\in L^2(\Omega,\operatorname{Isom}(M))$ and $\gamma\in\operatorname{Isom}(L^2(\Omega,M))$. By Theorem~\ref{isometriesM}, $\gamma$ is of the form $\gamma=\gamma^\varphi \, \gamma_\rho$ for some $\rho\in L^2(\Omega,\operatorname{Isom}(M))$ and $\varphi\in\operatorname{Aut}(\Omega)$. Clearly, $\gamma^{-1}=\gamma_{\rho^{-1}}\,\gamma^{\varphi^{-1}}$ and therefore 
$\gamma \gamma_\tau \gamma^{-1} = \gamma^\varphi \gamma_\rho \gamma_\tau \gamma_{\rho^{-1}} \gamma^{\varphi^{-1}} =\gamma_{\sigma},$ where $\sigma\in L^2(\Omega,\text{Isom}(M))$ is defined by $\sigma(\omega):=\rho(\varphi(\omega))\circ \tau(\varphi(\omega)) \circ \rho(\varphi(\omega))^{-1}$. Hence, $L^2(\Omega,\operatorname{Isom}(M))$ is normal.

By Theorem~\ref{isometriesM}, any isometry $\gamma\in\operatorname{Isom}(L^2(\Omega,M))$ can be written as $\gamma=\gamma^\varphi\circ\gamma_\rho=\gamma_{\rho\circ \varphi}\circ\gamma^\varphi$ for some $\rho\in L^2(\Omega,\operatorname{Isom}(M))$ and $\varphi\in\operatorname{Aut}(\Omega)$. Thus, 
$\operatorname{Isom}(L^2(\Omega,M))=\operatorname{Aut}(\Omega) L^2(\Omega,\operatorname{Isom}(M))=L^2(\Omega,\operatorname{Isom}(M))\cdot\operatorname{Aut}(\Omega).$
	
	Finally, we show $L^2(\Omega,\operatorname{Isom}(M))\cap\operatorname{Aut}(\Omega)=\{\operatorname{id}\}$. Indeed, if $\gamma\in L^2(\Omega,\operatorname{Isom}(M))\cap\operatorname{Aut}(\Omega)$, then $\gamma=\gamma_\rho=\gamma^\varphi$ for some $\rho\in L^2(\Omega,\operatorname{Isom}(M))$ and $\varphi\in\operatorname{Aut}(\Omega)$. By considering the action of this isometry on constant functions, we see that $\rho=\operatorname{id}_M$ and hence $\gamma=\operatorname{id}$.

Therefore, $\operatorname{Isom}(L^2(\Omega,M))\cong L^2(\Omega,\operatorname{Isom}(M))\rtimes\operatorname{Aut}(\Omega)$.
\end{proof}

\begin{proof}[Proof of Theorem~\ref{MN}]
	
For $\dim(M), \dim(N)\geq 2$, the claim follows directly from Theorem~\ref{isometriesM}.

The remaining spaces, $L^2(\Omega, S^1)$ and $L^2(\Omega, \mathbb{R})$, are not isometric due to their differing diameters: finite for the former, infinite for the latter.

Since $M$ isometrically embeds into $L^2(\Omega, M)$, if $L^2(\Omega, M) \cong L^2(\Omega, \mathbb{R})$, then $M$ would isometrically embed into a Hilbert space, implying $M \cong \mathbb{E}^n$. Thus, $L^2(\Omega, \mathbb{R})$ is not isometric to $L^2(\Omega, M)$ for any $M$ of dimension at least two with irreducible universal covering. 

Finally we also show that $L^2(\Omega, S^1)$ is not isometric to $L^2(\Omega, M)$ for the same $M$ as above. To that end, we first observe that for $f\in L^2(\Omega,S^1)$, there exists $\phi\in L^2(\Omega,\R)$ such that $f(\omega)=e^{i\phi(\omega)}$ for all $\omega\in\Omega$. The set $C=\{\alpha\in L^2(\Omega,\R): |\alpha|\leq\pi/2\}$ is convex and $\psi: C\to L^2(\Omega,S^1)$, given by $\alpha\mapsto e^{i(\phi+\alpha)}$, is an isometric embedding. 

We claim that for any geodesic $\sigma:[0,d] \to L^2(\Omega,S^1)$ starting at $f$, the first half of its image lies in $\psi(C)$. By Theorem~\ref{thm: monodgeodesics}, $\sigma(t)(\omega)=e^{i(\varphi(\omega)+t\alpha(\omega) )}$ for some $ \alpha\in L^2(\Omega,\R)$ with $\| \alpha\|_{L^2}=1$. Since $\alpha(\omega)d\leq \pi$ for all $\omega$, the claim follows.

Thus, for $f \in L^2(\Omega, S^1)$ and geodesics $\sigma_i : [0, d_i] \to L^2(\Omega, S^1)$ issuing from $f$, $i = 1, 2$, the convex hull of the geodesic triangle with vertices $f, \sigma_1(d_1/2), \sigma_2(d_2/2)$ isometrically embeds into an inner product space.

Analogously to above, if $L^2(\Omega, S^1) \cong L^2(\Omega, M)$, then $M$ isometrically embeds into $L^2(\Omega, S^1)$. Hence, for any pair of tangent vectors at a point in $M$, we can find a small flat triangle spanned by geodesics pointing into these directions. Hence, $M$ is flat. Therefore, $L^2(\Omega, S^1) \not\cong L^2(\Omega, M)$ for $\dim(M) \geq 2$ and $M$ with irreducible universal cover. This concludes the proof.
\end{proof} 

\subsection{Necessity of irreducibility assumption}\label{nes}
Let $X$ and $Y$ be metric spaces. The map $f \mapsto (f_1, f_2)$ is an isometry from $L^2(\Omega, X \times Y)$ to $L^2(\Omega, X) \times L^2(\Omega, Y)$, where $f_1$ and $f_2$ denote the projections of $f$ onto $X$ and $Y$, respectively. This, together with Theorem~\ref{MN}, demonstrates the necessity of some irreducibility condition for Theorem~\ref{L2factors1} to hold. 

The same applies to Theorem~\ref{MN} as illustrated by the following. 

\begin{lemma}\label{isomo}
Let $(\Omega, \mu)$ be an atomless standard probability space, and $X$ be a metric space. Then: $L^2(\Omega, X^n) \cong L^2(\Omega, \sqrt{n}X)$.
\end{lemma}

\begin{proof} By Theorem~\ref{Rokhlin}, $\Omega \cong [0,1]$. Thus, by Lemma~\ref{lem:cv}, $L^2(\Omega,X^n)\cong L^2([0,1],X^n)$ and $L^2(\Omega,\sqrt{n}X)\cong L^2([0,1],\sqrt{n}X)$. 

    Define $\gamma: L^2([0,1],X^n) \to L^2([0,1],\sqrt{n}X)$ by $\gamma(f)(t) = f_i(nt-i+1)$ for $t\in [\frac{i-1}{n},\frac{i}{n})$, where $f = (f_1, ..., f_n)$. This map is evidently surjective. The map is also isometric. Indeed for $f,f^\prime\in L^2([0,1],X^n)$,
    
    \begin{align*}
        d_{L^2([0,1],X^n)}^2(f,f^\prime) &= \int_0^1 \sum_{i=1}^n d_X^2(f_i(t),f_i^\prime(t)) \,dt \\
        &= \sum_{i=1}^n \int_{\frac{i-1}{n}}^{\frac{i}{n}} nd_X^2 (f_i(nt-i+1), f_i^\prime(nt-i+1)) \, dt \\
        &= \int_0^1 (\sqrt{n}d_X)^2(\gamma(f)(t),\gamma(f^\prime)(t))\,dt \\&= d^2_{L^2([0,1],\sqrt{n}X)}(\gamma(f),\gamma(f^\prime)).
    \end{align*}
    
This completes the proof, establishing $L^2(\Omega, X^n) \cong L^2(\Omega, \sqrt{n}X)$. 
\end{proof}

Finally the following example directly demonstrates the necessity of the irreducibility assumption for Theorem~\ref{isometriesM} and thereby in particular, Theorem~\ref{semi-main}. The example should give another direct sense of why these theorems fail for reducible spaces. 

\begin{example}\label{R1}
	
Let $M$ satisfy the assumptions of Theorem~\ref{isometriesM}. We construct an isometry of $L^2([0,1], M \times M)$ violating the rigidity behaviour of Theorem~\ref{isometriesM}. 

We will be using notations from the proof of Theorem~\ref{semi-main} above. Further let $\gamma: L^2([0,1], M\times M) \to L^2([0,1], \sqrt{2}M)$ be the isometry from the proof of Lemma~\ref{isomo} (for $n = 2, X = M$).

Consider $\varphi \in \text{Aut}([0,1])$ fixing $[0,1/4] \cup [3/4,1]$ and swapping $[1/4,1/2]$ and $[1/2,3/4]$. Let $(x, y) \in M \times M$. Then $(\gamma^\varphi \circ \gamma)(f_{((x,y))}^{[0,1]}) = f_{(x, y)}^{(A, A^c)}$ for $A = [0,1/4] \cup [1/2,3/4]$, and so $(\gamma^{-1} \circ \gamma^\varphi \circ \gamma)(f_{((x,y))}^{[0,1]}) = f_{((x, x), (y, y))}^{([0,1/2], [1/2,1])}$.

Since $\text{Aut}([0,1])$ acts trivially on constant functions, if the isometries were rigid in the sense of Theorem~\ref{isometriesM}, there would exist $\rho \in L^2([0,1], \text{Isom}(M \times M))$ such that $\gamma^{-1} \circ \gamma^\varphi \circ \gamma = \gamma_\rho$. Therefore, there would exist $\omega\in [0,1/2]$ such that the isometry $\rho(\omega):M\times M \to M \times M$ sends $(x, y)$ to $(x, x)$ for all $(x,y)\in M\times M$.  This is a contradiction.
\end{example}

As noted in the introduction, without the irreducibility assumption, we can still provide weaker algebraic characterizations of the isometry group.
\begin{remark}\label{rem}
	For atomless $\Omega$ and non-isometric complete Riemannian manifolds $M,N$ with irreducible universal covers, we can prove that isometries of $L^2(\Omega,M) \times L^2(\Omega,N)$ preserve the product structure. The arguments are similar to the arguments used for establishing Theorem~\ref{isometriesM}. 

For a complete simply connected Riemannian manifold $M$, let $\R^{m_0} \times M_1^{m_1}\times M_n^{m_n}$ be its de Rham decomposition, where $M_1,...,M_n$ are pairwise non-isometric, simply connected and irreducible Riemannian manifolds.  Thus by Lemma~\ref{isomo} and the just mentioned splitting of isometries of $L^2(\Omega,M_i) \times L^2(\Omega,M_j)$, we obtain:
\[
\operatorname{Isom}(L^2(\Omega,M)) \cong \operatorname{Isom}(L^2(\Omega,\R)) \times \prod_{i=1}^n L^2(\Omega,\operatorname{Isom}(M_i)) \rtimes \operatorname{Aut}(\Omega).
\]

Unlike Theorem~\ref{semi-main}, this is merely an abstract group isomorphism and does not specify explicit embeddings of the right-hand-side subgroups into the isometry group of $L^2(\Omega,M)$.
\end{remark}

\section*{Acknowledgments}  I would like to thank Alexander Lytchak for insightful comments and suggestions on this project. I would also like to thank Nicola Cavallucci for reading an earlier version of this work and providing valuable feedback, and James Terrill for helpful comments. Finally, I extend my sincere appreciation to the anonymous referee for an exceptionally thorough review and helpful suggestions, which enhanced the clarity and presentation of this work. This research was partially supported by the Deutsche Forschungsgemeinschaft (DFG,
German Research Foundation) under project number 281869850.
\newpage
\bibliographystyle{alpha}
\bibliography{paper}

\end{document}